\documentclass[11pt,a4paper,reqno]{amsart}
\pdfoutput=1

\usepackage[selected=true,shrink=15,stretch=30,letterspace=30]{microtype}
\usepackage{amsfonts}
\usepackage{amssymb}
\usepackage[utf8]{inputenc}
\usepackage[british]{babel}
\usepackage[margin=2cm]{geometry}
\usepackage{setspace}
\usepackage[url=false,eprint=false,isbn=false,giveninits=true]{biblatex}
\usepackage{xcolor}

\usepackage{ifpdf}
\ifpdf
\usepackage[unicode=true,pdftex,colorlinks=true, pdfborderstyle={/S/D/D[2 1]/W .5},linkcolor=blue,citecolor=violet,urlcolor=black,pdfversion=1.5,pdflang=en-GB
,pdfauthor={Alexander Dyachenko and Mikhail Tyaglov},pdfsubject={MSC(2010) Primary 15A18, 47B36; Secondary 34A05,  34L10, 34L05},%pdfpagemode=UseOutlines,%pdfa=true
,pdftitle={Linear differential operators with polynomial coefficients generating generalised Sylvester-Kac matrices}
,pdfkeywords={Sylvester-Kac matrix,  Eigenvalues,  Eigenvectors, Linear differential operators}
]{hyperref}
\else
\usepackage[unicode=true,colorlinks=true,linkcolor=blue,citecolor=violet,urlcolor=black]{hyperref}
\fi

\numberwithin{equation}{section}

\theoremstyle{plain}
\newtheorem{theorem}{Theorem}[section]

\theoremstyle{definition}
\newtheorem{remark}[theorem]{Remark}

\title[Linear differential operators generating generalised Sylvester-Kac matrices]{Linear
    differential operators with polynomial coefficients generating generalised Sylvester-Kac
    matrices}

\author[A.~Dyachenko]{Alexander Dyachenko}
\email[A.~Dyachenko]{diachenko@sfedu.ru}
\address{\smaller Keldysh Institute of Applied Mathematics, Russian Academy of Sciences, Miusskaya sq.~4, 125047 Moscow,  Russia}
\address{\smaller Department of Mathematics and Statistics, University of Konstanz, 78457 Konstanz, Germany}

\author[M.~Tyaglov]{Mikhail Tyaglov}
\email[M.~Tyaglov]{tyaglov@mail.ru}
\address{\smaller Moscow Center for Fundamental and Applied Mathematics, Moscow 119991, Russia}

\date{8 July 2021}

\AtBeginBibliography{\footnotesize}
\bibliography{General_Kac_Sylvester_arx2}

\begin{document}

\begin{abstract}
    A method of generating differential operators is used to solve the spectral problem for a
    generalisation of the Sylvester-Kac matrix. As a by-product, we find a linear differential
    operator with polynomial coefficients of the first order that has a finite sequence of
    polynomial eigenfunctions generalising the operator considered by M.\,Kac. In addition, we
    explain spectral properties of two related tridiagonal matrices whose shape differ from our
    generalisation.

    \bigskip
    

    \medskip

\end{abstract}

\maketitle

%%%%%%%%%%%%%%%%%%%%%%%%%%%%%%%%%%%%%%%%%%%%%%%%%%%%%%%%%%%%%%%%%%%%%%%%%%%%%%%%%%%%%%%%%%%%%%%
\section{Sylvester-Kac-type matrices: historical remarks and applications}\label{section:intro}
%%%%%%%%%%%%%%%%%%%%%%%%%%%%%%%%%%%%%%%%%%%%%%%%%%%%%%%%%%%%%%%%%%%%%%%%%%%%%%%%%%%%%%%%%%%%%%%

A matrix of the form
\begin{equation}\label{Sylvester.Kac.Matrix}
K_N:=\begin{pmatrix}
  0   &   1 &  0   &\dots&   0    &   0    &   0   \\
N &   0   &  2 &\dots&   0    &   0    &   0   \\
   0   & N-1 &   0   &\dots&   0    &   0    &   0   \\
\vdots&\vdots&\vdots&\ddots&\vdots&\vdots&\vdots\\
   0   &   0   &   0   &\dots& 0 &   N-1 &0\\
   0   &   0   &   0   &\dots&2&   0    &N\\
   0   &   0   &   0   &\dots&0&1 & 0    \\
\end{pmatrix}
\end{equation}
is called the \textit{Sylvester-Kac matrix}. First time, it appeared in
an extremely short paper by J.\,Silvester~\cite{Sylvester} in 1854.
Sylvester gave its characteristic polynomial without a proof. According to T.\,Muir's
fundamental work on the history of determinants,
the first proof of Sylvester's claim was provided by F.\,Mazza in 1866~\cite[p.~442]{Muir.II}.

In XX century, the Sylvester matrix got a new life and many applications as well as the
second name, the Kac matrix. M.\,Kac~\cite{Kac} being  not aware of Sylvester's work
found the spectrum of the matrix~\eqref{Sylvester.Kac.Matrix} and its eigenvectors
by the method of generating functions. Later on, this matrix and its certain generalisations
appeared in many publications. It was rediscovered many times by many authors and by
different approaches, see~\cite{Rosza,Clement,Vincze,Edelman_Kostlan.1994}. O.\,Taussky and
J.\,Todd~\cite{TausskyTodd} gave an account of various linear algebra approaches to the study of the
Sylvester-Kac matrix and its generalisation.

% Do we want to cite Chremmos_Efremidis ?
Also, matrix~\eqref{Sylvester.Kac.Matrix} and its generalisations found applications in such areas as orthogonal polynomials~\cite{Askey}, linear algebra~\cite{Ikramov,Holtz.det,Chu_Wang,Bevilacqua.Bozzo},
physics~\cite{Abdel-Rehim,%Chremmos_Efremidis,
Fonseca_et_al}, graph theory~\cite{Brouwer_Cohen_Neumaier}, numerical
analysis~\cite{Clement,Fernando,Munarini_Torri}, statistics~\cite{Edelman_Kostlan.1994,Edelman_et_al},
statistical mechanics~\cite{Kac,Siegert,Hess},
 biogeography~\cite{Igelnik_Simon} etc., see, e.g.,~\cite{Fonseca_et_al} for more references.

%The papers~\cite{Schroedinger,Askey,Holtz.det,Calogero_et_al_1,Chu_Wang,Chu1,Chu2,Vaia.Spadini} study
The papers~\cite{Askey,Holtz.det,Chu_Wang,Chu1,Fonseca_et_al} study 
spectral properties of various Sylvester-Kac-type matrices. The present
paper revisits and generalises some of their results % of~\cite{Askey,Holtz.det,Chu_Wang,Chu1,Fonseca_et_al} 
by using a different approach. In fact, R.\,Askey~\cite{Askey}
adopted the orthogonal polynomial approach and dealt with the Krawtchouk polynomials to prove some his
results we cover here. O.\,Holtz used matrix block-triangularisation to obtain the same results
as R.\,Askey. W.\,Chu~\cite{Chu1} employed the so-called left eigenvector method to find eigenvalues
of the matrix we consider here. However, he did not find its eigenvectors. Finally, the authors
of the works~\cite{Chu_Wang,Fonseca_et_al} guessed their results and proved that their guess is
correct by direct substitutions. Our approach is more constructive.

Indeed, we consider a linear differential operator of the first order with polynomial
coefficients. Its specialisation with an infinite sequence of polynomial solutions may be
transformed into another operator of a similar kind that has an infinite sequence of rational
eigenfunctions -- some of which (finitely many) are polynomials. As a result, we obtain a linear differential operator with polynomial coefficients having a finite sequence of polynomial eigenfunctions. M.\,Kac~\cite{Kac} came to a particular case of such an operator by the method of generating functions starting from the Sylvester-Kac matrix.
In turn, our starting point is the differential operator, and we arrive at a generalisation of the the Sylvester-Kac matrix. In fact, being restricted to the space $\mathbb{C}_{N}[z]$ of all complex polynomials
of degree at most~$N$, our operator becomes finite-dimensional, and its matrix representation is
a generalised Sylvester-Kac matrix. In this way, we obtain eigenvalues and eigenvectors of
this matrix.

The paper is organised as follows. Section~\ref{section:diff.op} is devoted to a first-order linear differential operator with polynomial coefficients.
We find its eigenvalues and the corresponding eigenfunctions: they turn to be rational functions including a prescribed number of polynomials.
%This operator turns to only have rational eigenfunctions including a prescribed number of polynomials. We find them and the corresponding eigenvalues.
In Section~\ref{section:Sylvester.matrix} we represent the restriction of the aforementioned differential operator to the space $\mathbb{C}_N[z]$ of all complex polynomials of degree at most $N$, \ $N\in\mathbb{N}$, as finite tridiagonal matrix.
% In Section~\ref{section:Sylvester.matrix} we show that the aforementioned differential operator being restricted to the space $\mathbb{C}_N[z]$ of all complex polynomials of degree at most $N$, $N\in\mathbb{N}$, has a matrix representation, and this matrix is a finite tridiagonal matrix.
Then we determine the eigenvalues and eigenvectors of that matrix, which is a generalisation of
the Sylvester-Kac matrix~\eqref{Sylvester.Kac.Matrix} depending on 4 parameters.
Section~\ref{section:particular.cases} deals with particular cases of our matrix. We show that
our results cover and generalise the results of some previous publications. In
Section~\ref{section:discussion}, we discuss some future applications of our approach. Finally,
Appendix~\ref{section:appendix} provides two distinct simple proofs of the
main\footnote{Other results of~\cite{FonsecaKilic,FonKilPer20} are substantially generalised
    in~\cite{DyTy}.}
results of~\cite{FonsecaKilic,FonKilPer20} concerned with matrices similar
to~\eqref{Sylvester.Kac.Matrix}. The first proof shows that these two publications, in fact,
rediscovered certain properties of persymmetric matrices for the particular case of the
Sylvester-Kac matrix. The other proof establishes a connection to tridiagonal matrices related
to the Hahn polynomials; it additionally produces formulae for the left and right eigenvectors
of the matrices studied in~\cite{FonsecaKilic,FonKilPer20}.

%how two recent works concerned with matrices close to~\eqref{Sylvester.Kac.Matrix}, in fact, (re)discovered certain properties of larger classes of matrices. In particular, we provide two simple proofs for the results of works~\cite{FonsecaKilic,FonKilPer20}, and generalise and complement those results.

%%%%%%%%%%%%%%%%%%%%%%%%%%%%%%%%%%%%%%%%%%%%%%%%%%%%%%%%%%%%%%%%%%%%%%%%%%%%%%%%%%%%%%%%%%%%%%%%%%%%%%%%%%%%%%
\section{Spectral problem for  differential operators with polynomial coefficients}\label{section:diff.op}
%%%%%%%%%%%%%%%%%%%%%%%%%%%%%%%%%%%%%%%%%%%%%%%%%%%%%%%%%%%%%%%%%%%%%%%%%%%%%%%%%%%%%%%%%%%%%%%%%%%%%%%%%%%%%%

Consider the differential operator
\begin{equation}\label{diff.op.classic}
L u(x)=x\,\dfrac{du(x)}{dx}
\end{equation}
acting in the space $\mathcal{S}$ of all formal power series of the form
\begin{equation}\label{formal.series}
\sum\limits_{m=-\infty}^{+\infty}a_mx^m,\quad a_k\in\mathbb{C}.
\end{equation}
It is easy to check that the eigenvalue problem
\begin{equation}\label{EV.problem.1}
L u=\lambda u,\qquad u\in\mathcal{S},
\end{equation}
has the following solutions
\begin{equation}\label{diff.op.classic.EV.sol}
\lambda_j=j,\quad u_j(x)=x^j,\qquad j=0,\pm1,\pm2,\ldots.
\end{equation}
Note that for $j\geqslant0$, the eigenfunctions $u_j(x)$ are polynomials, while
for $j<0$ they are rational functions with a unique pole of order~$-j$ at the origin.

The operator $L$ is a particular (singular) case of a more general
operator of the form
\begin{equation}\label{diff.op.classic.2.general}
\mathcal{L}u(z)=
(a+b z+c z^2)\dfrac{du(z)}{dz}+h zu(z),
\end{equation}
where $a,b,c,h\in\mathbb{C}$. However, it turns out that
the eigenvalues and eigenfunctions of~$\mathcal{L}$ in the space of formal power
series~\eqref{formal.series} can be found for certain $h$ by changing
variables in the eigenvalue problem~\eqref{EV.problem.1}.

Indeed, let us consider the eigenvalue problem~\eqref{EV.problem.1}
%
%\begin{equation*}
%x\dfrac{du(x)}{dx}=\lambda u(x), \qquad u(x)\in\mathcal{S},
%\end{equation*}
%
and make the following change of the variable
\begin{equation}\label{frac.lin.change.var}
x:=\dfrac{\alpha+\beta t}{\gamma +\delta t},\qquad \alpha\delta-\beta\gamma\neq0,
\end{equation}
that implies
\begin{equation*}
t=-\dfrac{\alpha-\gamma x}{\beta-\delta x}.
\end{equation*}
At the same time, given a fixed integer $N\geqslant1$ we also change the function~$u$ by introducing a new function
\begin{equation}\label{frac.lin.change.func}
w(t):=(\gamma+\delta t)^Nu(x),
\end{equation}
so that
\begin{equation*}%\label{frac.lin.change.func}
u(x)=\dfrac{w(t)}{(\gamma+\delta t)^N}.
\end{equation*}
This gives us
\begin{equation*}
x\dfrac{du(x)}{dx}=-\dfrac{(\alpha+\beta t)(\gamma+\delta t)}{\alpha\delta-\beta\gamma}\cdot\dfrac{d}{dt}
\left[\dfrac{w(t)}{(\gamma+\delta t)^N}\right]=
-\dfrac{\alpha+\beta t}{\alpha\delta-\beta\gamma}\cdot
\dfrac{(\gamma+\delta t)\,\dfrac{dw(t)}{dt}-N\delta w(t)}{(\gamma+\delta t)^{N}}.
\end{equation*}
%
%On multiplying by the denominator and collecting the terms,
Consequently,
the problem~\eqref{EV.problem.1}
transforms into a new eigenvalue problem
\begin{equation}\label{EV.problem.2}
\mathcal{L}_Nw=\mu w,\qquad w\in\mathcal{S},\quad N\in\mathbb{N},
\end{equation}
where
\begin{equation}\label{diff.op.classic.2.general.2}
\mathcal{L}_Nw(t)=(\alpha+\beta t)(\gamma+\delta t)\dfrac{dw(t)}{dt}-\beta\delta Ntw(t),\qquad N\in\mathbb{N},
\end{equation}
and
\begin{equation*}
\mu=\alpha\delta N-\lambda\mathcal{D}\qquad\text{with}\quad\mathcal{D}=\alpha\delta-\beta\gamma\neq0.
\end{equation*}

Now from~\eqref{diff.op.classic.EV.sol}, \eqref{frac.lin.change.var}, and \eqref{frac.lin.change.func} we obtain that the solutions of the eigenvalue problem~\eqref{EV.problem.2}--\eqref{diff.op.classic.2.general.2} are the following rational functions
\begin{equation}\label{frac.lin.change.operator.r=1.eigenfunc}
w_j(t)=(\alpha+\beta t)^j(\gamma+\delta t)^{N-j},\qquad j\in\mathbb{Z},
\end{equation}
corresponding to the eigenvalues
\begin{equation*}\label{frac.lin.change.operator.r=1.eigenval}
\mu_j=\alpha\delta N-\mathcal{D}j,\qquad j\in\mathbb{Z},\qquad\text{with}\quad\mathcal{D}=\alpha\delta-\beta\gamma\neq0.
\end{equation*}
%
%where
%
%\begin{equation*}\label{lin.frac.discr}
%\mathcal{D}=ad-bc\neq0.
%\end{equation*}
%
\begin{remark}\label{remark.1}
The formula~\eqref{frac.lin.change.operator.r=1.eigenfunc} shows that for $j=0,1,\ldots,N$, the eigenvalue problem~\eqref{EV.problem.2} has polynomial eigenfunctions $w_j$. All other eigenfunctions of~\eqref{EV.problem.2} are rational functions.
\end{remark}

%%%%%%%%%%%%%%%%%%%%%%%%%%%%%%%%%%%%%%%%%%%%%%%%%%%%%%%%%%%%%%%%%%%%%%%%%%%%%%%%%%%%%%%%%%%%%%% 
\section{Spectral problem for generalised Sylvester-Kac matrix}\label{section:Sylvester.matrix}
%%%%%%%%%%%%%%%%%%%%%%%%%%%%%%%%%%%%%%%%%%%%%%%%%%%%%%%%%%%%%%%%%%%%%%%%%%%%%%%%%%%%%%%%%%%%%%%

Let $\mathbb{C}_{N}[z]$, $N\in\mathbb{N}$, be the set of all polynomials with complex coefficients of degree at most $N$. It is well known that~$\mathbb{C}_{N}[z]$  is an $(N+1)$-dimensional space isomorphic to the space $\mathbb{C}^{N+1}$.

The operator $L$ defined in~\eqref{diff.op.classic} being restricted to $\mathbb{C}_{N}[z]$ has exactly $N+1$ polynomial eigenfunctions in the space $\mathbb{C}_{N}[z]$  for any $N\in\mathbb{N}$.
Remark~\ref{remark.1} says that the operator $\mathcal{L}_N$ defined in~\eqref{diff.op.classic.2.general.2} also has exactly $N+1$ eigenpolynomials. Therefore, we can restrict this operator to
$\mathbb{C}_{N}[z]$, and, in this space, $\mathcal{L}_N$ has exactly $N+1$ distinct eigenvalues and the correspondent polynomial eigenfunctions.

Let
\begin{equation}\label{Operator.restriction}
\mathcal{A}_N=\mathcal{L}_N\Bigg|_{\mathbb{C}_N[z]}.
\end{equation}
From~\eqref{diff.op.classic.2.general.2}, it follows that if
\begin{equation}\label{poly}
p(z)=a_0+a_1z+a_2z^2+\cdots+a_Nz^N\in\mathbb{C}_N[z],
\end{equation}
then
\begin{equation*}
(\alpha+\beta z)(\gamma+\delta z)\dfrac{dp(z)}{dz}-N\beta \delta p(z)=
\left[Na_N(\alpha\delta+\beta\gamma)-\beta\delta\cdot a_{N-1}\right]z^N+O\left(z^{N-1}\right)
\quad\text{as}\quad z\to\infty,
\end{equation*}
so
$\mathcal{L}_N p\in\mathbb{C}_N[z]$ for any $p\in\mathbb{C}_N[z]$.  Thus, we have
\begin{equation*}
\mathcal{A}_N:\mathbb{C}_N[z]\to\mathbb{C}_N[z].
\end{equation*}
%
%and if $p\in\mathbb{C}_N[z]$, $\deg p=N$, then $\deg\left(\mathcal{A}_Np\right)\leqslant N$.

Consequently, $\mathcal{A}_N$ is a finite-dimensional operator, and the eigenvalue problem
\begin{equation*}
\mathcal{A}_Nv=\mu v
\end{equation*}
has exactly $N+1$ linearly independent polynomial eigenfunctions
\begin{equation}\label{frac.lin.change.operator.eigenfunc}
w_j(z)=(\alpha+\beta z)^j(\gamma+\delta z)^{N-j},\qquad j=0,1,\ldots,N,
\end{equation}
corresponding to the eigenvalues
\begin{equation}\label{frac.lin.change.operator.eigenval}
\mu_j=\alpha\delta N-\mathcal{D}j,\qquad j=0,1,\ldots,N,\qquad\text{with}\quad\mathcal{D}=\alpha \delta-\beta\gamma\neq0.
\end{equation}

On the other hand, the operator $\mathcal{A}_N$ can be represented as an $(N+1)\times(N+1)$ matrix. Namely, for the polynomial~$p$ defined by~\eqref{poly}, let us consider the (column) vector $v=(a_0,a_1,\ldots,a_N)^{T}$ of its coefficients (here~``\,$T$\,'' stands for the transpose). Then there exists a matrix $J_N$ such that
$J_N\,v=(b_0,b_1,\ldots,b_N)^{T}$ is the vector of the coefficients of the polynomial $\mathcal{A}_Np$. From~\eqref{diff.op.classic.2.general.2},
\eqref{Operator.restriction}, and~\eqref{poly}, one gets
\begin{equation*}
\begin{array}{l}
b_0=\alpha\gamma\cdot a_1,\\[5pt]
b_1=-N\beta\delta\cdot a_0+(\alpha\gamma+\beta\delta)a_1+2\alpha\gamma\cdot a_2,
\\[5pt]
\cdots\cdots\cdots\cdots\cdots\cdots\cdots\cdots\cdots\cdots\cdots\cdots\cdots\cdots\cdots\cdots\cdots\cdots
\\[5pt]
b_k=-(N-k+1)\beta\delta\cdot a_{k-1}+k(\alpha\gamma+\beta\delta)a_k+(k+1)\alpha\gamma\cdot a_{k+1},
\\[5pt]
\cdots\cdots\cdots\cdots\cdots\cdots\cdots\cdots\cdots\cdots\cdots\cdots\cdots\cdots\cdots\cdots\cdots\cdots
\\[5pt]
b_{N-1}=-2\beta\delta\cdot a_{N-2}+(N-1)(\alpha\gamma+\beta\delta)a_{N-1}+N\alpha\gamma\cdot a_{N},
\\[5pt]
b_{N}=-\beta\delta\cdot a_{N-1}+N(\alpha\gamma+\beta\delta)a_{N}.
\end{array}
\end{equation*}

Thus, the matrix
\begin{equation}\label{diff.op.classic.2.matrix.general.2}
J_{N}=
\begin{pmatrix}
  0   &   \alpha\gamma &  0   &\dots&   0    &   0    &   0   \\
-N\beta\delta\!\! &   \alpha\delta+\beta\gamma   &  2\alpha\gamma &\dots&   0    &   0    &   0   \\
   0   & \!\!-(N-1)\beta\delta &  \!2(\alpha\delta+\beta\gamma)\!\!   &\dots&   0    &   0    &   0   \\
%   0   &0 &   (N-2)\gamma   &   3\beta   &\dots&   0    &   0    &   0   \\
\vdots&\vdots&\vdots&\ddots&\vdots&\vdots&\vdots\\
%   0   &   0   &   0   &   0   &\dots&(N-3)\alpha&   (N-3)(N-2)&0\\
   0   &   0   &   0   &\dots& \!\!(N-2)(\alpha\delta+\beta\gamma)\!\!\! &   (N-1)\alpha\gamma &0\\
   0   &   0   &   0   &\dots&-2\beta\delta&  \!\!(N-1)(\alpha\delta+\beta\gamma)\!\!\!  &N\alpha\gamma\\
   0   &   0   &   0   &\dots&0&-\beta\delta & \!\!N(\alpha\delta+\beta\gamma)    \\
\end{pmatrix},
\end{equation}
is the matrix representation of the operator $\mathcal{A}_N$ in $\mathbb{C}^{N+1}$ in the canonical basic. Consequently, $J_N$ has
the eigenvalues~\eqref{frac.lin.change.operator.eigenval}, and the correspondent eigenvectors
are the vectors of the coefficients of the polynomials~\eqref{frac.lin.change.operator.eigenfunc}.
We therefore arrive at the following theorem.
\begin{theorem}
Under the conditions that~$\alpha$, $\beta$, $\gamma$, $\delta\neq0$ and~$\alpha\delta\neq \beta\gamma$, the eigenvalues of the matrix $J_N$ defined by~\eqref{diff.op.classic.2.matrix.general.2}
are
\begin{equation}\label{frac.lin.change.matrix.eigenval}
\mu_j=\alpha\delta(N-j)+\beta\gamma\cdot j,\qquad j=0,1,\ldots,N,
\end{equation}
and $v_j=(v_{0j},v_{1j},\ldots,v_{Nj})^{T}$ is the eigenvector corresponding to $\mu_j$, where
\begin{equation}\label{frac.lin.change.matrix.eigenvectors}
v_{kj}=\sum\limits_{i=0}^{\min(k,j)}\binom{j}{i}\binom{N-j}{k-i}\left(\dfrac{\delta}{\gamma}\right)^{k-i}\left(\dfrac{\beta}{\alpha}\right)^{i},\qquad k=0,1,\ldots,N.
\end{equation}
\end{theorem}
\begin{proof}
The formula~\eqref{frac.lin.change.matrix.eigenval} follows from~\eqref{frac.lin.change.operator.eigenval}.
The formula~\eqref{frac.lin.change.matrix.eigenvectors} follows from the fact that the eigenpolynomials~\eqref{frac.lin.change.operator.eigenfunc}
of the operator $\mathcal{A}_N$ can be represented in the form
\[
    w_j(z)=(\alpha+\beta z)^j(\gamma+\delta z)^{N-j}
    % =\left(\sum_{i=0}^j\binom{j}{i}\alpha^{j-i}(\beta z)^i\right)
    % \left(\sum_{m=0}^{N-j}\binom{N-j}{m}\gamma^{N-j-m}(\delta z)^{m}\right)
    =\alpha^j\gamma^{N-j}\sum_{i=0}^j\sum_{m=0}^{N-j}\binom{j}{i}\binom{N-j}{m}
    \left(\frac\beta\alpha\right)^{i}\left(\frac\delta\gamma\right)^{m} z^{i+m}
    % =\alpha^j\gamma^{N-j}\sum_{k=0}^{N}\sum_{i=0}^{\max\{k,j\}}\binom{j}{i}
    % \binom{N-j}{k-i}\left(\frac\beta\alpha\right)^{i}\left(\frac\delta\gamma\right)^{k-i} z^{k}
    ,
\]
which after a change of the summation index turns into
\begin{equation*}
\frac{w_j(z)}{\alpha^j\gamma^{N-j}}=\sum\limits_{k=0}^N\sum\limits_{i=0}^{\min(k,j)}\binom{j}{i}\binom{N-j}{k-i}\left(\dfrac{\delta}{\gamma}\right)^{k-i}\left(\dfrac{\beta}{\alpha}\right)^{i}z^k=\sum\limits_{k=0}^Nv_{kj}z^k.
\end{equation*}
\end{proof}

\begin{remark}
The case when at least one of the numbers $\alpha$, $\beta$, $\gamma$, $\delta$ equals zero (with $\alpha\delta-\beta\gamma\neq0$) is not very interesting
from the matrix point of view, since the matrix~\eqref{diff.op.classic.2.matrix.general.2} is triangular in this case.

Regarding the differential operator $\mathcal{L}_N$ defined in~\eqref{diff.op.classic.2.general.2}, for~$\beta=0$ or $\delta=0$ it degenerates (up to a linear change of the variable) to the operator~$L$ of the form~\eqref{diff.op.classic}. The case $\alpha=0$ or $\gamma=0$ with $\beta\delta\neq0$ can be transformed by a linear change of the variable into the generic case when none of the numbers $\alpha,\beta,\gamma,\delta$ in the operator~$\mathcal{L}_N$ is zero.
\end{remark}

\begin{remark}\label{rem.degenerate.case}
If~$\mathcal{D}=\alpha\delta-\beta\gamma=0$, we cannot use the linear-fractional transform as in~\eqref{frac.lin.change.var}. The operator~$\mathcal{L}_N$ defined by~\eqref{diff.op.classic.2.general.2} then has cases depending on whether~$\beta\delta=0$ or not.

If~$\beta\delta=0$, then (unless~$\mathcal{L}_N$ is trivial) the condition~$\mathcal{D}=0$ implies\footnote{Otherwise, the leading polynomial coefficient of the operator~\eqref{diff.op.classic.2.general.2} becomes zero.} that~$\beta=\delta=0$, and hence
\[
\mathcal{L}_Nw(z)=\alpha\gamma\dfrac{dw(z)}{dz}.
\]
Here the only eigenpolynomial is~$w_0(z)\equiv1$, and the corresponding eigenvalue is~$\mu_0=0$. In this case, the matrix of the operator~$\mathcal{A}_N$ has one nontrivial diagonal, namely the superdiagonal; the unique eigenvalue~$\mu_0$ of~$\mathcal{A}_N$ is of algebraic multiplicity $N+1$ and of geometric multiplicity~$1$.

If $\beta\delta\ne0$, then~$(\gamma+\delta z)=\frac{\delta}{\beta}(\alpha+\beta z)$, and hence
\[
\mathcal{L}_Nw(z)
=\frac{\delta}{\beta}(\alpha+\beta z)^2\dfrac{dw(z)}{dz}-\beta\delta Nzw(z).
\]
So, on letting~$t=(\alpha+\beta z)$ and~$p(t)=w(z)$ the eigenproblem~$\mathcal{L}_Nw(z)=\mu w(z)$ transforms into
\[
%\mathcal{L}_Nw(z)-\alpha\delta N w(z)=
\delta t^2\dfrac{dp(t)}{dt}-\delta N t p(t) =(\mu-\alpha\delta N) p(t)
.
\]
An examination of the coefficients of this equality near the highest and lowest powers of~$t$ shows that it may only hold when~$\deg p =N$, and only when~$\mu=\alpha\delta N$. However, these two restrictions imply that~$p(t)=t^N$ up to a normalisation. Accordingly, the only eigenpolynomial of~$\mathcal{L}_N$ in this case is~$w_0(z)=(\alpha+\beta z)^N$, which corresponds to the eigenvalue~$\mu_0=\alpha\delta N$. The matrix of the operator~$\mathcal{A}_N$ also has a unique eigenvalue~$\mu_0$ of algebraic multiplicity $N+1$ and of geometric multiplicity~$1$. The characteristic polynomial of $J_N$ for the specific case~$\alpha=-\beta=-1/2$ and~$\gamma=-\delta=1$ was found by L.\,Painvin in 1858, see~\cite[p.~434]{Muir.II}.
\end{remark}

%%%%%%%%%%%%%%%%%%%%%%%%%%%%%%%%%%%%%%%%%%%%%%%%%%%%%%%%%%%%%%%%%%%%%%%%%%%%% 
\section{Particular cases}\label{section:particular.cases}
%%%%%%%%%%%%%%%%%%%%%%%%%%%%%%%%%%%%%%%%%%%%%%%%%%%%%%%%%%%%%%%%%%%%%%%%%%%%%

\noindent In this section, we consider particular cases of the matrix $J_N$ defined in~\eqref{diff.op.classic.2.matrix.general.2}.

Given~$a,b,c\in\mathbb{C}$, ~$c\neq0$, let us set
\begin{equation}\label{change.variable}
\alpha:=\dfrac{b-\sqrt{D}}{4c},\quad \beta:=\dfrac12,\quad \gamma:=b+\sqrt{D},\quad \delta:=2c,\quad\text{where}\quad
D=b^2-4ac.
\end{equation}
Then the matrix~\eqref{diff.op.classic.2.matrix.general.2} gets the form
\begin{equation}\label{diff.op.classic.2.matrix.general}
B_{N}(a,b,c)=
\begin{pmatrix}
  0   &   a &  0   &\dots&   0    &   0    &   0   \\
-Nc &   b   &  2a &\dots&   0    &   0    &   0   \\
   0   & -(N-1)c &   2b   &\dots&   0    &   0    &   0   \\
%   0   &0 &   (N-2)c   &   3b   &\dots&   0    &   0    &   0   \\
\vdots&\vdots&\vdots&\ddots&\vdots&\vdots&\vdots\\
%   0   &   0   &   0   &   0   &\dots&(N-3)a&   (N-3)(N-2)&0\\
   0   &   0   &   0   &\dots& (N-2)b &   (N-1)a &0\\
   0   &   0   &   0   &\dots&-2c&   (N-1)b    &Na\\
   0   &   0   &   0   &\dots&0&-c & Nb    \\
\end{pmatrix}.
\end{equation}
It represents the differential operator
\begin{equation}\label{diff.op.classic.2.general.1}
L_{a,b,c}u(z)=
(a+b z+c z^2)\dfrac{du(z)}{dz}-Nc zu(z)
\end{equation}
restricted to $\mathbb{C}_N[z]$.

\begin{remark}\label{Remark.non.zero.coeff.r=1}
In~\eqref{diff.op.classic.2.general.1} we additionally suppose that $a\neq0$, since the case $a=0$ can be transformed into the generic case ($a\neq0$) by a linear change of the variable $z$.
\end{remark}

The expressions~\eqref{frac.lin.change.matrix.eigenval}--\eqref{frac.lin.change.matrix.eigenvectors} and~\eqref{change.variable} imply that the matrix~\eqref{diff.op.classic.2.matrix.general} has the following  eigenvalues:
\begin{equation}\label{diff.op.classic.2.eigenvalues.general}
\lambda_{j}=j\cdot\dfrac{b+\sqrt{b^2-4ac}}{2}+(N-j)\cdot\dfrac{b-\sqrt{b^2-4ac}}{2},\quad j=0,1,\ldots,N,
\end{equation}
and the correspondent eigenvectors $v_j=(v_{0j},v_{1j},\ldots,v_{Nj})^{T}$ are given by
\begin{equation}\label{diff.op.classic.2.solution.coeff.general}
v_{kj}=\left(\dfrac{2c}{b+\sqrt{D}}\right)^k\cdot
\sum\limits_{i=0}^{\min(k,j)}\binom{j}{i}\binom{N-j}{k-i}\left(\dfrac{b+\sqrt{D}}{b-\sqrt{D}}\right)^i,\quad j=0,1,\ldots,N,
\end{equation}
where $D$ is defined in~\eqref{change.variable}.

The (rational) eigenfunctions of the operator~\eqref{diff.op.classic.2.general.1} in the space $\mathcal{S}$ corresponding the
eigenvalues~\eqref{diff.op.classic.2.eigenvalues.general} for $j\in\mathbb{Z}$ are the following

\begin{equation}\label{diff.op.classic.2.solution.general}
Q_j(z)=\dfrac{(2c)^N}{\big(\sqrt{D}-b\big)^j\big(\sqrt{D}+b\big)^{N-j}}\left(z-\dfrac{\sqrt{D}-b}{2c}\right)^{j}
\left(z+\dfrac{\sqrt{D}+b}{2c}\right)^{N-j}, \qquad j\in\mathbb{Z}.
\end{equation}

\vspace{3mm}

Let us list some particular cases of the matrix $B_{N}(a,b,c)$ considered considered in literature.

\begin{itemize}
\item[1)] The case $b=0$, $a=-c=1$ or $\alpha=\beta=\gamma=1$, $\delta=-1$, corresponds to the Sylvester-Kac matrix~\cite{Kac,Chu_Wang,Askey,TausskyTodd,Holtz.det,Muir.II,Rosza,Vincze}.
\item[2)] According to T.\,Muir~\cite[p.~434]{Muir.II}, the case $b=1$, $a+c=1$ or $\alpha\delta+\beta\gamma=\alpha\gamma+\beta\delta=1$ was first considered by L.\,Painvin in 1858 for eigenvalues (see also~\cite{Askey,Holtz.det}). W.\,Chu and X.\,Wang~\cite{Chu_Wang} found eigenvectors for this matrix.
\item[3)] The case $a=1-p$, ~$b=2p-1$, ~$c=-p$ or $\alpha\delta+\beta\gamma=2p-1=-(\alpha\gamma+\beta\delta)$ (up to a transposition and a shift of eigenvalues) is
related to the Krawtchouk polynomials~\cite{Askey,Holtz.det}.
The corresponding eigenvectors were found in~\cite{Chu_Wang}.
\item[4)] The eigenvalues and eigenvectors for the case $b=-(c+a)$ or $\alpha\delta+\beta\gamma=-(\alpha\gamma+\beta\delta)$ (up to a shift of eigenvalues) were
found in~\cite{Fonseca_et_al}. This case covers the case $3)$. Note that the characteristic polynomial
of this matrix (up to a diagonal shift) was found by T.\,Muir~\cite[\S~576]{Muir}.
\item[5)] The eigenvalues of the matrix~\eqref{diff.op.classic.2.matrix.general} for arbitrary $a$, $b$, and $c$ were found in~\cite{Chu1}. The eigenvectors~\eqref{diff.op.classic.2.solution.coeff.general} of the matrix $B_{N}(a,b,c)$ are new.
\end{itemize}

As we mentioned in Section~\ref{section:intro}, all techniques in the aforementioned works are different from the one used here.
Thus, we generalise the results of the works~\cite{Askey,Holtz.det,Chu_Wang,Chu1,Fonseca_et_al} in a simple and a unified way.

Note that in the degenerated case $b^2=4ac$ (i.e. $D=0$) the matrix $B_{N}(a,b,c)$ has a unique eigenvalue with exactly one
eigenvector. In this case, the operator~\eqref{diff.op.classic.2.general.1} restricted to $\mathbb{C}_N[z]$ also has only
one eigenvalue with a unique polynomial eigenfunction for every fixed $N\in\mathbb{N}$, cf. Remark~\ref{rem.degenerate.case}.

%%%%%%%%%%%%%%%%%%%%%%%%%%%%%%%%%%%%%%%%%%%%%%%%%%%%%%%%%%%%%%%%%%%%%%%%%%%%%%
\section{Discussion}\label{section:discussion}
%%%%%%%%%%%%%%%%%%%%%%%%%%%%%%%%%%%%%%%%%%%%%%%%%%%%%%%%%%%%%%%%%%%%%%%%%%%%%

%We considered\ldots

The method applied in the present paper can be used to find the eigenvalues and eigenvectors of the tridiagonal
matrix whose entries are the recurrence relation coefficients for the Hahn polynomials. It was noticed by A.\,Kova\v{c}ec~\cite{Kovacec} that the
spectrum of this matrix was conjectured by E.\,Schr\"odinger in~\cite{Schroedinger}. A.\,Kova\v{c}ec gave a proof of Schr\"odinger's conjecture~\cite{Kovacec}. However,
R.\,Askey~\cite{Askey} and O.\,Holtz~\cite{Holtz.det} proved\footnote{In two different ways distinguished from Kova\v{c}ec's one.} this conjecture much earlier, while
W.\,Chu and X.\,Wang~\cite{Chu_Wang} found eigenvectors of the corresponding matrix. 
R.\,Oste and J.\,Van der Jeugt~\cite{OsteJeugt_2016,OsteJeugt_2017} recovered the results of R.\,Askey and O.\,Holtz with their own original method, and gave an orthogonal polynomial interpretation of the results of W.\,Chu and X.\,Wang.

Our approach allows us
to find the eigenvalues and eigenvectors of the Schr\"odinger matrix and to solve the generalised
eigenvalue problem for a pair of linear differential operators in a very simple and more constructive manner. We believe that the
results~\cite{Calogero_et_al_1,Chu2,Vaia.Spadini,OsteJeugt_2016,OsteJeugt_2017} can also be improved by a similar approach, but this study is to be a subject for another paper.

There is a number of related problems where our approach does not give an immediate result,
although other techniques prove to be very efficient. Appendix~\ref{section:appendix} of this
work employs two distinct methods to explain why the spectra of two astonishingly simple
tridiagonal matrices distinct from~\eqref{Sylvester.Kac.Matrix} are integer. These matrices were
studied in works~\cite{FonsecaKilic,FonKilPer20}, but we believe that our
Appendix~\ref{section:appendix} provides deeper understanding of their properties.

Finally, some other matrices related to the Sylvester-Kac matrix~\eqref{Sylvester.Kac.Matrix} are studied
in~\cite{Kilic2013,KilicArikan}. Our work~\cite{DyTy} not only finds spectra of those matrices
in a straightforward way, but also determines their eigenvectors. In fact, \cite{DyTy} expresses
solution to the eigenvalue problem for a general tridiagonal matrix with 2-periodic
main diagonal via the spectral data of the same matrix, in which the main diagonal is put to
zero.

%%%%%%%%%%%%%%%%%%%%%%%%%%%%%%%%%%%%%%%%%%%%%%%%%%%%%%%%%%%%%%%%%%%%%%%%%%%%%
\section{Acknowledgement}
%%%%%%%%%%%%%%%%%%%%%%%%%%%%%%%%%%%%%%%%%%%%%%%%%%%%%%%%%%%%%%%%%%%%%%%%%%%%%

The authors thank C.\,da Fonseca for pointing out to the works~\cite{Kilic2013,KilicArikan,FonsecaKilic,FonKilPer20}. A.\,Dyachenko was in part supported by the research fellowship DY 133/2-1 of the German Research Foundation~(DFG).
%The work of M.\,Tyaglov was partially supported by
%National Natural Science Foundation of China under grant no.\ 11871336.

\appendix

\section{Concerning two special matrices studied by da Fonseca et al.}\label{section:appendix}

The authors of~\cite{FonsecaKilic} calculate the eigenvalues of the matrix\footnote{We notice that this matrix is a particular case of the one appeared in the work of A.\,Caley~\cite{Caley.1858}, see also~\cite[p.~429]{Muir.II} and~\cite[p.~355]{TausskyTodd}.}
\begin{equation}\label{Matrix.G}
G_N=
\begin{pmatrix}
   0   &   1   &0&\dots&   0    &   0    &   0   \\
  2N+2 &   0   &2&\dots&   0    &   0    &   0   \\
   0   & 2N+1  &0&\dots&   0    &   0    &   0   \\
\vdots&\vdots&\vdots&\ddots&\vdots&\vdots&\vdots\\
   0   &   0   &0&\dots& 0 &   N-1 &0\\
   0   &   0   &0&\dots&N+4&  0    &N\\
   0   &   0   &0&\dots&0&N+3 & 0    \\
\end{pmatrix}
\end{equation}
via a smart choice of the basis and an induction in size of the matrix. In an analogous way, the
work~\cite{FonKilPer20} deals with the matrix~$H_N=\frac 12 S_N$, where
\begin{equation}\label{Matrix.S}
S_N=
\begin{pmatrix}
   0   &   1   &0&\dots&   0    &   0    &   0   \\
  2N &   0   &2&\dots&   0    &   0    &   0   \\
   0   & 2N-1  &0&\dots&   0    &   0    &   0   \\
\vdots&\vdots&\vdots&\ddots&\vdots&\vdots&\vdots\\
   0   &   0   &0&\dots& 0 &   N-1 &0\\
   0   &   0   &0&\dots&N+2&  0    &2N\\
   0   &   0   &0&\dots&0&N+1 & 0
\end{pmatrix}
.
\end{equation}
The matrices~$G_N$ and~$S_N$ are the $(N+1)\times(N+1)$ leading
principal submatrices of the Sylvester-Kac matrix~\eqref{Sylvester.Kac.Matrix}: of~$K_{2N+2}$ and~$K_{2N}$, respectively, -- except that the~$(N,N+1)$th entry of~$S_N$ is doubled. It
turns out~\cite{FonsecaKilic,FonKilPer20} that the spectra\footnote{Here,~$\sigma(A)$ stands for the spectrum of a matrix~$A$. We remind the reader that in~\cite{FonKilPer20} the authors deal with $H_N=\frac 12 S_N$.} of these matrices are integer:
% \begin{alignat}{3}\label{eq:G_n_spectrum}
%     \sigma(G_{2n-1})&=\{\pm2(2k-1)\}_{k=1}^n,&&
%     \sigma(G_{2n})&=\{\pm4k\}_{k=0}^n,
%     \\
%     \label{eq:H_n_spectrum}
%     \sigma(H_{2n-1})&=\{\pm2(2k-1)\}_{k=1}^n
%     \quad\text{and}\quad&&
%     \sigma(H_{2n})&=\{\pm4k\}_{k=0}^n
%     .
% \end{alignat}
\begin{alignat}{3}\label{eq:G_n_H_n_spectrum}
    \sigma(S_{N})=\sigma(G_{N})&=\left\{2(2j-N)\right\}_{j=0}^N
    %\quad\text{and}\quad&&
    %\sigma(S_{2m})=\sigma(G_{2m})&=\left\{\pm4k\right\}_{k=0}^m
    .
\end{alignat}
However, the method of~\cite{FonsecaKilic,FonKilPer20} provides little understanding of the
phenomenon. There is another reasoning relying on the properties of persymmetric Jacobi
matrices, which is in our opinion both shorter and deeper. A matrix is called persymmetric if
all its entries are symmetric with respect to the anti-diagonal.

Let~$D_{N}$ denote the~$(N+1)\times(N+1)$ diagonal matrix of the factorials,
$D_{N}=\big\|j!\,\delta_{i,j}\big\|_{i,j=0}^{N}$. Here $\delta_{i,j}$ stands for the Kronecker delta symbol. Direct computation shows that the
product~$D_{2N+2}K_{2N+2}D_{2N+2}^{-1}$ is a persymmetric matrix with a unit superdiagonal% , whose
% spectrum is the same as the spectrum of~$A$
. Moreover, its $(N+1)\times(N+1)$ leading principal submatrix is the matrix~$D_N G_N D_N^{-1}$ whose
spectrum coincides with the spectrum of~$G_N$. At the same time, \cite[eq.~(15),
Lemma~3.3]{GenestEtAl} implies that the spectrum of~$D_N G_N D_N^{-1}$ contains every second
spectral point of~$D_{2N+2}K_{2N+2}D_{2N+2}^{-1}$, or more specifically that~$\sigma(G_{N})=\sigma(D_{N}G_{N}D_{N}^{-1})$ satisfies~\eqref{eq:G_n_H_n_spectrum}.
%for both cases~$N=2m-1$ and~$N=2m$.

Now, let~$I_N$ denote the~$(N+1)\times(N+1)$ identity matrix. Consider two polynomials
\[
    \Omega_1(x)=\det(xI_N-G_N) \quad\text{and}\quad
    \Omega_0(x)=\frac{\det(xI_{2N+2}-K_{2N+2})}{\Omega_1(x)}.
\]
From~\cite[eq.~(15), Lemma~3.3]{GenestEtAl} we additionally have
\[
    \det(xI_{N+1}-M_{N+1})  %=P_{n+2}(x)
    =\frac{\Omega_0(x)+x\Omega_1(x)}{2},
\]
where~$M_{N+1}$ is the $(N+2)\times(N+2)$ leading principal submatrix of~$K_{2N+2}$. Now, using
linearity of~$\det(xI_{N+1}-S_{N+1})$ with respect to the last column we obtain:
\[
    \det (xI_{N+1}-S_{N+1})
    =
    2 \det(xI_{N+1}-M_{N+1}) %P_{n+2}(x)
    -x \Omega_1(x)
    =
    \Omega_0(x)+x\Omega_1(x)-x\Omega_1(x)=\Omega_0(x),
\]
which yields~\eqref{eq:G_n_H_n_spectrum} for~$\sigma(S_{N})$. Thus, the main results
of~\cite{FonsecaKilic,FonKilPer20} rediscover some properties of persymmetric Jacobi matrices for the particular case of the Sylvester-Kac matrix.

\smallskip

On the other hand, the main results of~\cite{FonsecaKilic,FonKilPer20} may also be derived from the results of the works~\cite{Holtz.det,Askey,Chu_Wang}. Indeed, let us consider the matrix of the three-term recurrence relations for the Hahn polynomials with $\alpha=\beta$ 
(see~\cite[formula~(4.9)]{Askey}):
\begin{equation*}
C_N(\alpha)=
\begin{pmatrix}
b_0&a_0&0&\cdots&0&0\\
c_1&b_1&a_1&\cdots&0&0\\
0&c_2&b_2&\cdots&0&0\\
\vdots&\vdots&\vdots&\ddots&\vdots&\vdots\\
0&0&0&\cdots&b_{N-1}&a_{N-1}\\
0&0&0&\cdots&c_N&b_N\\
\end{pmatrix},    
\end{equation*}
where
\begin{equation*}
\begin{array}{ll}
a_0=\dfrac N2,\ \ a_i=\dfrac{(i+2\alpha+1)(N-i)}{2(2i+2\alpha+1)},\quad & i=1,\ldots,N-1,\\
\\
b_i=\dfrac{N}{2}, & i=0,1,\ldots,N,\\
\\
c_i=\dfrac{i(i+2\alpha+N+1)}{2(2i+2\alpha+1)}, & i=1,\ldots,N,
\end{array}
\end{equation*}
As is shown in~\cite{Askey,Holtz.det}, the spectrum of the matrix $C_N(\alpha)$ does not depend on $\alpha$ and has the form
\begin{equation}\label{Spectrum.C}
\sigma\big(C_N(\alpha)\big)=\{0,1,\ldots,N\}.
\end{equation}

With $S_N$ defined by~\eqref{Matrix.S}, the aforementioned matrix $H_N=\frac12 S_N$ is related to the matrix $C_N\left(-\frac12\right)$ as follows:
\begin{equation}\label{Relation.C.H}
H_N^{T}=E_N\left[2C_N\left(-\frac12\right)-N\cdot I_N\right]E_N^{-1},
\quad\text{where}\quad
E_N=\|\delta_{i,N-j}\|_{i,j=0}^{N}.
\end{equation}
In~\cite[Theorem~5.1]{Chu_Wang}, the authors found the eigenvectors for the matrix of the three-term recurrence relations corresponding to the Racah polynomials, which turn into the (right) eigenvectors of~$C_N(\alpha)$ on putting $\alpha=\beta$ and $\gamma\to\infty$. According to~\eqref{Relation.C.H}, reversing the order of entries in these vectors yields the transposed left eigenvectors of the matrix $H_N$. In other words, from~\eqref{Spectrum.C}--\eqref{Relation.C.H} and~\cite[Theorem~5.1]{Chu_Wang} for $\mathbf{u}_j H_N=\lambda_j \mathbf{u}_j$ we directly obtain
\begin{equation}\label{Sylvester-Kac.spectrum}
\lambda_j=2j-N,\qquad j=0,1,\ldots,N,
\end{equation}
with the entries of the row-vector 
$\mathbf{u}_j=(u_{0j},u_{1j},\ldots,u_{Nj})$ given by the following formula\footnote{We put $\alpha=\beta=-\frac12$ and $\gamma\to\infty$ in~\cite[Theorem~5.1]{Chu_Wang} and multiply the matrix of the right eigenvectors of $C\left(-\frac12\right)$ by the matrix $E_N$.}
\begin{equation*}\label{Matrix.H.left.E.Vect}
u_{kj}
=\sum\limits_{i=0}^{\min\{N-k,j\}}
  (-1)^{k+i}\binom{N-k}{i}\binom{N-i}{j-i}
  \cdot\dfrac{(N-k)_i}{(1/2)_i}
%=\sum\limits_{i=0}^{\min\{N-k,j\}}
%  (-1)^{k+i}\binom{N-i}{j-i}
%  \cdot\dfrac{(N-k)\cdot (N+1-k-i)_{2 i-1}}{i!(1/2)_i}
% the second formula is correct, but there is a problem for k=N.
  ,\quad k,j=0,1,\ldots,N,
\end{equation*}
where $(a)_i=a(a+1)\cdots(a+i-1)$. Moreover, there is a simple formula expressing the right eigenvectors of a tridiagonal matrix via its left eigenvectors, see e.g.~\cite{DyTy}. Thus, the entries of the right eigenvector~$\mathbf{v}_j=(v_{0j},v_{1j},\ldots,v_{Nj})^T$ 
of $H_N$ corresponding to the $j$th 
eigenvalue~\eqref{Sylvester-Kac.spectrum}, $j=0,1,\ldots,N$, have the form
\begin{equation*}\label{Matrix.H.right.E.Vect}
    \begin{aligned}
        v_{kj}&=\binom{2N}{k}\sum\limits_{i=0}^{\min\{N-k,j\}}(-1)^{k+i}\binom{N-k}{i}\binom{N-i}{j-i}\cdot\dfrac{(N-k)_i}{(1/2)_i},\quad k=0,1,\ldots,N-1
        ,\\
        v_{Nj}&=\frac 12\binom{2N}{N}\binom{N}{j}.
    \end{aligned}
\end{equation*}

The matrix $G_N$ defined in~\eqref{Matrix.G} is related to the matrix $C_N\left(\frac12\right)$ in a similar manner, namely,
\begin{equation}\label{Relation.C.G}
G_N^{T}=R_N \left[4 C_N\left(\frac12\right)-2N I_N\right] R_N^{-1}
\end{equation}
with~$R_N=\|(i+1)\delta_{i,N-j}\|_{i,j=0}^{N}$. Now from~\eqref{Spectrum.C} and~\eqref{Relation.C.G} we obtain that the eigenvalues of $G_N$ are
\begin{equation}\label{Sylvester-Kac.spectrum2}
\lambda_j=2(2j-N), \quad j=0,1,\ldots,N.
\end{equation}
On letting $\alpha=\beta=\frac12$ and $\gamma\to\infty$ in~\cite[Theorem~5.1]{Chu_Wang}, we see that for~$j=0,1,\ldots,N$ the left eigenvector~$\mathbf{u}_j=(u_{0j},u_{1j},\ldots,u_{Nj})$ of~$G_N$ corresponding to the eigenvalue $\lambda_j$ has the following entries:
\begin{equation*}\label{Matrix.G.left.R.Vect}
u_{kj}
%=(N+1-k)\sum\limits_{i=0}^{\min\{N-k,j\}}
%  (-1)^{k+i}\binom{N-k}{i}\binom{N-i}{j-i}
%  \cdot\dfrac{(N+2-k)_i}{(3/2)_i}
=\sum\limits_{i=0}^{\min\{N-k,j\}}
  (-1)^{k+i}\binom{N-k}{i}\binom{N-i}{j-i}
  \cdot\dfrac{(N+1-k)_{i+1}}{(3/2)_i}
  ,\quad k=0,1,\ldots,N.
\end{equation*}
Consequently, for~$k,j=0,1,\ldots,N$ the $k$th entry of the right eigenvector
$\mathbf{v}_j=(v_{0j},v_{1j},\ldots,v_{Nj})^T$ corresponding to the $j$th 
eigenvalue~\eqref{Sylvester-Kac.spectrum2} has the form
\begin{equation*}\label{Matrix.G.right.R.Vect}
v_{kj}=\binom{2N+2}{k}\sum\limits_{i=0}^{\min\{N-k,j\}}
  (-1)^{k+i}\binom{N-k}{i}\binom{N-i}{j-i}
  \cdot\dfrac{(N+1-k)_{i+1}}{(3/2)_i}.
\end{equation*}

Thus, the main results of the works~\cite{FonsecaKilic,FonKilPer20} on the spectra of the matrices~$G_N$ and~$H_N$ also follow from the results of the works~\cite{GenestEtAl} or~\cite{Askey,Holtz.det}, while the formulae for the entries of the left and right eigenvectors of~$G_N$ and~$H_N$ were actually found in~\cite{Chu_Wang}.

\setstretch{1.0}
%\section{References}
\printbibliography\vspace{8pt}

\setlength{\parskip}{-7pt}
\end{document}